\documentclass[11pt]{amsart}
\input epsf
\usepackage{latexsym}
\usepackage[usenames,dvipsnames]{color}
\usepackage[colorlinks=true,linkcolor=blue,citecolor=BrickRed,urlcolor=RoyalBlue]{hyperref}
            
\usepackage{amssymb,amsmath,amsthm,amscd, amssymb,
graphicx, enumerate, verbatim, mathrsfs, color, url, bm, hyperref, parskip, centernot}
\usepackage[all]{xy}\usepackage[labelformat=empty]{caption}

\numberwithin{equation}{section}
\newtheorem{cor}[equation]{Corollary}
\newtheorem{lem}[equation]{Lemma}

\newtheorem{thm}[equation]{Theorem}

\newtheorem{Example}[equation]{Example}
\newenvironment{ex}{\begin{Example}\rm}{\end{Example}}
\newtheorem{remark}[equation]{Remark}
\newenvironment{rmk}{\begin{remark}\rm}{\end{remark}}
\def\co{\colon\thinspace}

\newcommand{\Int}{\mbox{Int}}
\newcommand{\Diff}{\mbox{Diff\hspace{1pt}}}

\newcommand{\e}{\varepsilon}

\def\a{\alpha}

\def\de{\delta}

\def\g{\gamma}

\def\Si{\Sigma}

\def\b{\beta}

\def\d{\partial}

\def\r{\rho}

\def\s{\sigma}

\def\Z{\mathbb{Z}}
\def\N{\mathbb{N}}
\def\S1{\bf S^1}
\newcommand{\cpn}{{\mathbb{C}\mathrm{P}^n}}
\newcommand{\hpn}{{\mathbb{H}\mathrm{P}^n}}
\newcommand{\cpk}{{\mathbb{C}\mathrm{P}^k}}
\newcommand{\hpk}{{\mathbb{H}\mathrm{P}^k}}

\newcommand{\cptwo}{{\mathbb{C}\mathrm{P}^2}}

\newcommand{\R}{{\mathbb R}}
\newcommand{\C}{{\mathbb C}}

\makeatletter
\def\equalsfill{$\m@th\mathord=\mkern-7mu
\cleaders\hbox{$\!\mathord=\!$}\hfill
\mkern-7mu\mathord=$}
\makeatother

\begin{document}

\abovedisplayskip=6pt plus3pt minus3pt
\belowdisplayskip=6pt plus3pt minus3pt

\title[Deformations of open nonnegatively curved manifolds]
{\bf Index theory and deformations of open nonnegatively curved manifolds}

\thanks{This work was partially supported by the Simons Foundation grant 524838.}
\rm
\author{Igor Belegradek}

\address{Igor Belegradek\\School of Mathematics\\ Georgia Institute of
Technology\\ Atlanta, GA 30332-0160}\email{ib@math.gatech.edu}

\date{}
\begin{abstract}
We use an index-theoretic technique of Hitchin to show that 
the space of complete Riemannian metrics of nonnegative sectional curvature
on certain open spin manifolds has nontrivial homotopy groups
in infinitely many degrees. A new ingredient of independent interest is homotopy density
of the subspace of metrics with cylindrical ends. 
\end{abstract}
\maketitle
\thispagestyle{empty}

\section{Introduction}

A cheap way to deform a Riemannian metric $g$  on an $n$-manifold $M$ is to pull it back
via diffeomorphisms of $M$. To simplify matters let us
replace $\Diff(M)$ by the subgroup of diffeomorphisms supported in a smoothly embedded 
$n$-disk in $M$. 
If $g$ enjoys some geometric property, one may ask whether the just described
$\Diff (D^n, \d)$-orbit map through $g$ is null-homotopic in the space of metrics 
with the property. The question was studied by N.~Hitchin in~\cite{Hit} and more
recently in~\cite{CroSch, CSS} in the setting when $M$ is a closed spin manifold and the property 
is invertability of the Dirac operator. These authors found order two
elements in certain homotopy groups of $\Diff (D^n, \d)$ 
that remain nontrivial under the above orbit map in 
any $\Diff(M)$-invariant space of Riemannian metrics on $M$ with invertible Dirac operators. 
Nontriviality is detected by the $\a$-invariant, the $KO$-valued index of the Dirac operator.

Here we apply Hitchin's method to 
(the compact-open smooth topology on) 
the space $\mathcal R_{K\ge 0}(V)$ of complete metrics of nonnegative sectional
curvature on an open connected manifold $V\!$. 
Any such $V$ is diffeomorphic to an open tubular neighborhood of a 
certain compact totally convex submanifold without boundary, called
a {\em soul\,}, which is determined 
by the metric and the basepoint in $V\!$, see~\cite{CheGro-soul}.

Given an element in a homotopy group of $\Diff (D^n, \d)$ with
nontrivial $\a$-invariant, we push it to $\mathcal R_{K\ge 0}(V)$ 
as in Hitchin's method, assume it is null-homotopic, deform the homotopy
so that the resulting family of metrics becomes cylindrical
of positive scalar curvature on an end of $V\!$, double the metrics along a cylinder 
cross-section, and conclude that the homotopy on the double runs through metrics 
with invertible Dirac operators, which gives a contradiction. This strategy works
under certain assumptions on $V$ resulting in the following theorem. 

\begin{thm}
\label{thm: intro-main}
Let $V$ be an open connected spin $n$-manifold with a complete metric $g$ of \mbox{$K\ge 0$}  
whose soul is not flat. Suppose 
one of the following holds:
\begin{itemize}
\item[(i)]
$V$ is the product of $\R$
and a closed smooth manifold,
\item[(ii)]
the normal sphere bundle to the soul of $(V, g)$ has no section, and 
for every metric in the path-component 
of $g$ in $\mathcal R_{K\ge 0}(V)$
the normal exponential map to the soul is a diffeomorphism. 
\end{itemize}
If $m\ge n\ge 6$ and  $m\equiv 0,1 (\mathrm{mod}\, 8)$, 
then the homotopy group
$\pi_{m-n}(\mathcal R_{K\ge 0}(V), g)$  has an element of order two.
\end{thm}

Results of~\cite{GuiWal-transitive, GroZil-milnor, GroZil-lift} 
and bundle-theoretic considerations
provide a number of examples where Theorem~\ref{thm: intro-main} applies.

\begin{cor} 
\label{cor: intro applications}
Let $L$ be a closed spin positive-dimensional manifold of $K\ge 0$,
and let $V$ be the total space of any of following bundles: \vspace{-1pt}
\begin{itemize}
\item[(1)] the trivial $\R$-bundle over $L$,\vspace{-1pt}
\item[(2)] the tangent bundle to $S^k$ for even $k>0$,\vspace{-1pt}
\item[(3)] any nontrivial $\R^3$-bundle over such $S^4$, $S^5$, $S^7$,\vspace{-1pt}
\item[(4)] any $\R^4$-bundle over $S^4$ with nonzero Euler class,\vspace{-1pt}
\item[(5)] the tautological
$\R$-bundle over $\mathbb{R}\mathrm{P}^{2k}$ for odd $k$,
\vspace{-1pt}
\item[(6)] the tangent bundles to $\cpk$ and $\hpk$ for even $k>0$,\vspace{-1pt}
\item[(7)] the tautological quaternionic line bundle over $\hpk$ with $k>0$,\vspace{-1pt} 
\item[(8)] any nontrivial $\R^2$-bundle over $\cpk$, $S^2\times S^2$, or $\cpk\#\pm \cpk$\vspace{-1pt}
with the same 2nd Stiefel--Whitney as the base, where $k>0$,\vspace{-2pt}
\item[(9)] any non-spin $\R^3$-bundle over $\cptwo$ with no nowhere zero section,\vspace{-2pt}
\item[(10)] 
the product of $L$ and any bundle in \textup{(2)--(9)} above.
\end{itemize}
If $m\ge n=\dim(V)\ge 6$ and $m\equiv 0,1\, (\mathrm{mod}\, 8)$, 
then for any $g\in\mathcal R_{K\ge 0}(V)$ there is an order two element in
$\pi_{m-n}(\mathcal R_{K\ge 0}(V), g)$.
\end{cor}

A key step in the proof is to deform the metric to one with
a cylindrical end. A Riemannian metric has a {\em cylindrical end\,} if the induced metric on the
complement of a compact codimension zero submanifold 
is isometric to the product of a closed manifold and $[0,\infty)$.
Let $\mathcal C_{K\ge 0}(V)$ denote the space of all metrics in $\mathcal R_{K\ge 0}(V)$
that have a cylindrical end. The two spaces of metrics
coincide if $V$ has a metric with codimension one soul.
L.\,Guijarro~\cite{Gui-improv} used a computation in~\cite{Kro} 
to show that any open complete manifold $(V, g)$ of $K\ge 0$ admits a metric $h$ with 
a cylindrical end such that $g=h$ on a tubular neighborhood of the soul. 
If the normal exponential map to the soul is a diffeomorphism, 
one can pick $h$ arbitrary close to $g$. 
A family version of this result holds if the normal bundle
to a soul is sufficiently twisted: 

\begin{thm}
\label{thm: push to C}
Let $V$ be an open $n$-manifold with a complete metric of $K\ge 0$ such that  
the normal sphere bundle to the soul of $(V, g)$ has no section and the normal 
exponential map to the soul is a diffeomorphism for every metric in $\mathcal R_{K\ge 0}(V)$.  
Then there is a homotopy $\r_\e$ of self-maps of
$\mathcal R_{K\ge 0}(V)$ such that $\r_0$ is the identity and 
for all $g\in \mathcal R_{K\ge 0}(V)$ and $\e\in (0, 1]$
the metric $\r_\e(g)$ lies in $\mathcal C_{K\ge 0}(V)$, and 
the $\r_\e(g)$-distance 
from the soul to a cylindrical end is at most $1+\frac{1}{\e}$. 
\end{thm}
Thus the homotopy $\r_\e$ instantly pushes $\mathcal R_{K\ge 0}(V)$ 
into $\mathcal C_{K\ge 0}(V)$, in which case one says that the latter
is {\em homotopy dense\,} in the former.
That ``the normal sphere bundle to the soul has no section''
ensures that the soul is uniquely determined by the metric, and
also depends continuously on the metric, see~\cite[Theorem 2.1]{BFK}, 
which leads to continuity of $\r_\e$. 

Let us sketch the proof of Theorem~\ref{thm: intro-main}. 
Fix an embedded $n$-disk in $M$ inside the \mbox{$1$-neighborhood} of the soul of $g$.
Start with
an order two element in the $k$th homotopy group of $\Diff (D^n, \d)$
that is detected by the $\a$-invariant. Represent the element by a map from the $k$-sphere,
and push it to $\mathcal R_{K\ge 0}(V)$ via the \mbox{$\Diff (D^n, \d)$-orbit} map through $g$. 
Suppose arguing by contradiction that this singular $k$-sphere contracts in 
$\mathcal R_{K\ge 0}(V)$, i.e., extends to a map from $D^{k+1}$ 
giving a family of metrics $g_y$ parametrized by $y\in D^{k+1}$.
Applying Theorem~\ref{thm: push to C} we can arrange all metrics
to have cylindrical ends that are within a definite distance to their souls.
For $r\gg 1$ we get a family of $r$-neighborhoods $N_y$ of the souls of $g_y$
such that $\d N_y$ lie in the cylindrical end of $N_y$, and in fact,
$\d N_y$ will be a cylinder cross-section, so that the induced metric on $\d N_y$
has $K\ge 0$. 
Contractibility of the disk allows to find the ambient isotopies $\phi_y$
with compact support moving $N_g$ to $N_y$. A further 
ambient isotopy arranges that the pullback
metrics $\phi_y^* g_y$ have the same product structure near $\d N_g$,
i.e., they are of the form $dr^2+b_y$ where $r$ is the distance to $\d N_g$.
We refer to the cylindrical region outside $N_g$ as the ``neck''.
We wish to modify the metric on the neck to make it of positive scalar curvature.
A result of C.~B\"ohm and B.~Wilking~\cite{BohWil} implies
that the Ricci flow on a closed non-flat manifold 
instantly turns a metric of $K\ge 0$ into a metric of positive scalar curvature. 
Applying Ricci flow to $b_y$ with running time depending of $r$ we get
a family of metrics that agrees with $dr^2+b_y$ for small $r$ and then has positive
Ricci curvature on the fibers of $r$. Since isotopy implies concordance
in the positive scalar curvature category, we can reparametrize the metric
to make it of positive scalar curvature on the neck.
Choosing the neck long enough we can arrange that its union $U_g$ with $N_g$ 
contains the support of every $\phi_y$. Pulling back the metrics via $\phi_y^{-1}$ we 
get a family of metrics on $U_g$ of nonnegative scalar curvature that is not identically zero.
Near $\d U_g$ the metrics are all equal, so they extend to the double of $U_g$. 
The result is a family of metrics parametrized by $y$ that has 
nonnegative scalar curvature that is not identically zero.
It is easy to see that double is spin if and only if $V$ is spin,
so the double has invertible Dirac operator. Since the given singular $k$-sphere 
has nonzero $\a$-invariant, it cannot contract through such metrics,
which is the desired contradiction. 

Let us briefly review
the previous works on topological properties of $\mathcal R_{K\ge 0}(V)$.
The results in~\cite{KPT, BKS-mod1, BKS-mod2, Ott-pairs, Ott-triv-norm, DKT, GAZ}
give various examples where $\mathcal R_{K\ge 0}(V)$, 
or even its $\Diff(V)$-quotient $\mathcal M_{K\ge 0}(V)$,
has infinitely many path-components. 
The geometric ingredient is that metrics in the same path-component 
have ambiently isotopic (and hence diffeomorphic) souls. 
Examples where $\mathcal R_{K\ge 0}(V)$
is shown to have finitely many nonzero rational higher homotopy groups  are given in~\cite{BFK}.
By contrast, Theorem~\ref{thm: intro-main} yields order two elements
of homotopy groups in infinitely many degrees.
Results in~\cite{TusWie} consider the case when $V$ has a
codimension one soul with a torus factor,
and uses it to give examples where $\mathcal M_{K\ge 0}(V)$ has nontrivial rational
homotopy and cohomology. Finally,~\cite{BelHu, BelHu-err, BanBel} determines the homeomorphism type 
of $\mathcal R_{K\ge 0}(\R^2)$, and also studies connectedness properties
of $\mathcal M_{K\ge 0}(\R^2)$.
There is a substantial recent literature on deformations of metrics 
subject to other curvature conditions, such as positive scalar curvature,
and as a starting point
the reader could consult~\cite{FT-ober}.

{\bf Convention:} in this paper {\em smooth\,} means $C^\infty$, and any set of 
diffeomorphisms, embeddings, submanifolds, or 
Riemannian metrics is equipped with the compact-open smooth topology. 

{\bf Acknowledgments:} I am grateful to Boris Botvinnik for leading me to~\cite{PetYun}.

{\bf Structure of the paper.} 
Section~\ref{sec: transitive} builds on results in~\cite{GuiWal-transitive}  
related to the condition (ii) of Theorem~\ref{thm: intro-main}, and needed
for Corollary~\ref{cor: intro applications};
in fact, the topic of Section~\ref{sec: transitive} may be of independent interest. 
Section~\ref{sec: family of hypersurfaces} contains a proof of Theorem~\ref{thm: push to C}.
The other results stated in the introduction are proved in Section~\ref{sec: main}.

\section{The normal exponential map to a soul is a diffeomorphism}
\label{sec: transitive}

We seek a topological condition on an open connected manifold $V$ ensuring that 
the normal exponential map $\exp^\bot\co\nu_S\to V$ to the soul $S$
of each metric in $\mathcal{R}_{K\ge 0}(V)$ is a diffeomorphism. 
A simple example of such condition is $H_{n-1}(V;\Z_2)\neq 0$, 
i.e., every soul have codimension one, which
by the splitting theorem~\cite{CheGro-soul} ensures that $\exp^\bot$
is a diffeomorphism.  
Another such condition was discovered by 
L.~Guijarro and G.~Walschap in~\cite{GuiWal-transitive}, and here we elaborate
and extend their work. 

By a standard argument the normal exponential map to a closed submanifold is a diffeomorphism
if and only if it each unit vector that is normal to the submanifold
exponentiates to a ray.
Since $V$ is open, each $x\in S$ is a starting point of a ray,
and since $S$ is a soul, the ray must be orthogonal to 
$S$~\cite[Theorem 5.1(3)]{CheGro-soul}.
The parallel transport along $S$ preserves the set of 
rays orthogonal to $S$~\cite[Lemma 1.1]{Wal-cd2}, and in particular,
the rays from $x$ are permuted
by the corresponding normal holonomy group $G_x$ of $\nu_S$ at $x$. Thus
{\em if $G_x$-action on the normal unit sphere at $x$ is transitive, then 
the normal exponential map is a diffeomorphism.}

At this point let us switch to the following more general setting.
Let $\xi$ be a smooth Euclidean vector bundle over a closed connected manifold $B$
(in fact, many results of this section also hold if $B$ is non-compact).
Let $p\co S(\xi)\to B$ be the unit sphere bundle. 
Equip $\xi$ with a metric connection, i.e., an (Ehresmann) connection whose parallel transport preserves
the Euclidean inner product, and let $G$ be the holonomy
group of the connection at $b\in B$. Then $G$ is a (possibly non-closed)
subgroup in the isometry group of the sphere $F=p^{-1}(b)$, which is isomorphic to $O(k)$
where $k$ is the dimension of the fibers of $\xi$. 
The identity component $G_0$ of $G$ is a closed subgroup of $O(k)$, and
the parallel transport defines a surjective homomorphism $\pi_1(B)\to G/G_0$,
so $G/G_0$ is countable~\cite[Theorem 4.2 in Chapter II]{KN-book}.

We say that 
{\em the holonomy of $\xi$ is transitive}
if the $G$-action on $F$ is transitive.

Consider the following conditions on $\xi$:
\begin{enumerate}
\item[(A)] $\pi_q(p)\co \pi_q(S(\xi))\to\pi_q(B)$ is not surjective for some $q\ge 1$,
\item[(B)] the holonomy of every metric connection on $\xi$ is transitive, 
\item[(C)]  $\xi$ does not split as a Whitney sum (of positive rank bundles),
\item[(D)]  $p$ has no section.
\end{enumerate}

The condition (B) is most relevant to this paper, while (D) is the easiest to check.
In favorable situations (A)--(D) are all equivalent, and in general, the conditions
help to isolate the cases where (B) hold. 
The following is contained in~\cite{GuiWal-transitive}:

\begin{thm} 
\label{thm: implications}
The implications  
$\mathrm{(A)}\Rightarrow\mathrm{(B)}
\Rightarrow\mathrm{(C)}\Rightarrow\mathrm{(D)}$ hold.
\end{thm}

\begin{proof}
The contrapositives of 
$\mathrm{(B)}\Rightarrow\mathrm{(C)}\Rightarrow\mathrm{(D)}$ are immediate: 
If $p$ has a section, then its span is a one-dimensional 
subbundle which together with its orthogonal complement gives a Whitney sum decomposition for $\xi$, and
if $\xi$ is a nontrivial Whitney sum, then any connections on the summands gives rise to a connection
on the Whitney sum whose holonomy violates (B). 
The implication $\mathrm{(A)}\Rightarrow\mathrm{(B)}$
is stated in~\cite[page 253]{GuiWal-transitive} when $q\ge 2$ and $\xi$ is orientable
but these assumptions are never used in the proof.
\end{proof}

\begin{rmk} 
\label{rmk: share total space} 
If (D) holds for a normal bundle
to a soul, then the soul is uniquely determined by the metric~\cite[Section 2]{BFK}.
The souls for different metrics need not even be diffeomorphic, but their normal sphere bundles
are fiber homotopy equivalent, and more generally,
if $\xi$, $\eta$ are smooth vector bundles over closed manifolds
with the same total space, then their
normal sphere bundles are fiber homotopy equivalent over the canonical homotopy equivalence
of their zero sections~\cite[Proposition 4.1]{BKS-mod1}.
Clearly, (A) and (D) are preserved under such fiber homotopy equivalences; 
I do not know if this holds for (B) or (C).  
\end{rmk}

\begin{lem} 
\label{lem: pullback}
The negation of each of the conditions \rm{(A)--(D)}\em\ is inherited 
by the pullback via a smooth map $f$ of base manifolds, i.e., if $\xi$ does
not satisfy the condition, then neither does $f^*\xi$.
\end{lem} 
\begin{proof}
Think of $f\co B^\prime\to B$ as the restriction of the deformation retraction $r$
of the mapping cylinder of $f$ onto $B$. If (A), (C), or (D) fails for $\xi$,
then it does for the pullback $r^*\xi$, and hence for its restriction to $B^\prime$.
If (B) fails for a connection on $\xi$, consider the pullback connection on $f^*\xi$.
For any two unit vectors in a fiber of $f^*\xi$ that are translates of
each other along a loop $\g$, their images in $\xi$ are translates of each other
along a loop $f\circ\g$.
\end{proof}

\begin{ex}
\label{ex: product times L}
For a (closed connected smooth) manifold $L$
consider the bundle $L\times\xi$ whose projection is the product of
the identity map of $L$ and the projection of $\xi$.
Then $\xi$ satisfies one of the conditions (A)--(D) if and only is so does $L\times\xi$.
because the bundles are pullbacks of each other.
\end{ex}

A partial converse of Lemma~\ref{lem: pullback} for the condition (A) 
comes from the following observation: 
{\em If both $f$ and the projection of the unit sphere bundle 
of $f^*\xi$ are $\pi_q$-surjective,
then so is $p$.} This is immediate from the map of the homotopy exact sequences 
of fiber bundles induced by $f$.

\begin{cor}
\label{cor: pullback under cover}
If $f\co B^\prime\to B$ is a covering map and $k\ge 2$, then
$\xi$ satisfies {\rm (A)} if and only if so does $f^*\xi$. 
If $f\co B^\prime\to B$ is a torus bundle and $k\ge 3$, then
$\xi$ satisfies {\rm (A)} if and only if so does $f^*\xi$. 
\end{cor}
\begin{proof}
For both claims the ``if'' direction follows by Lemma~\ref{lem: pullback}.
The converse of the first statement is true because $p$ is $\pi_1$-surjective as $k\ge 2$ 
so by assumption $p$ is not $\pi_q$-surjective for some $q\ge 2$, but
$f$ is a $\pi_q$-isomorphism for all $q\ge 2$.
Similarly, for the second statement from $k\ge 3$ we conclude that
$p$ is surjectve on $\pi_1$ and $\pi_2$, so by assumption 
$p$ is not $\pi_q$-surjective for some $q\ge 3$, but
$f$ is a $\pi_q$-isomorphism for all $q\ge 3$.
\end{proof}

\begin{ex}
If $V$ is an open complete manifold of $K\ge 0$, then
any path-component of $\mathcal{R}_{K\ge 0}(V)$ contains a metric
whose pullback to a finite Galois cover splits 
as the Riemannian product
of a torus and an open simply-connected complete  manifold $N$ of $K\ge 0$,
see~\cite[Corollary 6.3]{Wil-fund-gr}.
Now Lemma~\ref{cor: pullback under cover}, Example~\ref{ex: product times L}, 
and Remark~\ref{rmk: share total space}
imply that the normal bundle to a soul of $V$
satisfies (A) if and only if so does the normal bundle to a soul
of $N$. Thus in principle, {\em verifying} (A) {\em for normal bundles to souls reduces 
to the case of simply-connected manifolds of $K\ge 0$}. The caveat is that
in some cases $\xi$ may be easier to understand than $N$.
\end{ex}

\begin{ex}
The bound $k\ge 3$ in Corollary~\ref{cor: pullback under cover} 
cannot be improved: $TS^2$ satisfies (A) while its pullback under the Hopf fibration
$S^3\to S^2$ does not.
\end{ex}

\begin{ex}
\label{ex: 2-fold}
If $\xi$ is a nontrivial $\R$-bundle, then $\xi$ satisfies (A):
the nontrivial two-fold-cover $p$ is not $\pi_1$-surjective.
Of course, (D) fails for the trivial $\R$-bundle.
\end{ex}

\begin{rmk}
\label{rmk: euler hurewicz}
In many cases verifying (A) 
hinges on the following observation 
in~\cite{GuiWal-transitive}: 
the homotopy class of a map $f\co S^q\to B$
lies in the image of the $\pi_q(p)\co\pi_q(S(\xi))\to\pi_q(B)$ if and only if 
the pullback of the sphere bundle $p$ via $f$ has a section. 
Since every sphere bundle with a section has zero Euler class, we conclude:
\begin{center}
Condition (A) holds if the Euler class of $\xi$ (with $\Z$ or $\Z_2$ coefficients)\\
is nonzero on the image of the Hurewicz homomorphism.
\end{center}
Note that for an $(m-1)$-connected cell complex with $m\ge 2$ the Hurewicz homomorphism
is an isomorphism in degree $m$ and a surjection in degree 
$m+1$~\cite[Theorem 4.37, and Exercise 23 in section 4.2]{Hat-book}.
In~\cite[Theorem 1.4]{GSW} one finds the following version of the above:
(A) holds if the rational Euler class is nonzero,
$B$ is simply-connected and rationally $\frac{k}{2}$-connected.
\end{rmk}

\begin{rmk}
\label{rmk: sphere 1-4}
If $B$ is a sphere, then all the four conditions (A)--(D)
are equivalent because if the projection
of a sphere bundle over $S^q$ is $\pi_q$-surjective, then it has a section, which is the contrapositive
of $\mathrm{(D)}\Rightarrow\mathrm{(A)}$. More generally, by Example~\ref{ex: product times L}
we get:
{\em The conditions \rm (A)--(D)\ \em are equivalent if $\xi$ is the product of a closed 
smooth manifold $L$ and a vector bundle bundle over a sphere.}
\end{rmk}

Another natural way to obtain (B) is based on the fact that 
the structure group of $\xi$ reduces 
to $G$~\cite[Theorem 7.1 in Chapter II]{KN-book}, so the existence a connection with
non-transitive potentially restricts $\xi$, and here is how the 
idea can be exploited.

\begin{lem}
\label{lem: R3 bundles no section}
$\mathrm{(D)}\Longrightarrow\mathrm{(B)}$ for any 
$\R^3$-bundle $\xi$ with simply-connected $B$. 
\end{lem}
\begin{proof}
Since $B$ is simply-connected, the structure group of $\xi$ lies in $SO(3)$,
and the holonomy group of any Euclidean metric connection on $\xi$ is connected.
If (B) fails for $\xi$, then the holonomy group is a proper connected closed
subgroup of $SO(3)$, that is, the standard $SO(2)$. Hence $\xi$ has a section.
\end{proof}

\begin{ex}
\label{ex: SO(3) bundles over CP2}
By~\cite[Theorem 1]{GroZil-lift} every non-spin vector bundle over $\cptwo$
admits a complete metric of $K\ge 0$. Let us show that 
most non-spin $\R^3$-bundles over $\cptwo$ satisfy (B). 
To this end recall that by~\cite{DolWhi} the set of isomorphism clases 
of non-spin $\R^3$-bundles over $\cptwo$ is bijective 
to $\{1+4d\,:\, d\in\Z\}$, where $\xi$ is sento to the first 
Pontryagin class of $\xi$ evaluated on the fundamental class of $\cptwo$.
Such $\xi$ has a section  if and only $\xi\cong\e^1\oplus\eta$, 
the Whitney sum of a trivial $\R$-bundle 
and an $\R^2$-bundle $\eta$. This happens if and only if 
$p_1(\xi)$ is the square of $e(\eta)$, the Euler class of $\eta$.
Equivalently, $1+4d=k^2$ for some integer $k$. 
The letter happens if and only if $d$ is the product of two consecutive integers.
In summary, the non-spin $\R^3$-bundle over $\cptwo$ that corresponds to $1+4d$
satisfies (B) if and only if $d$ is not the product of two consecutive integers.
\end{ex}

The following lemma illustrates yet another method of checking (B).

\begin{lem}
\label{lem: TCPn THPn}
Let $L$ be a closed connected smooth manifold, and let $\xi$ be
a smooth Euclidean vector bundle over a closed manifold $B$.
If $\xi$ shares the total space with $L\times T\cpn$ or 
$L\times T\hpn$ for an even positive $n$, 
then $\xi$ satisfies \textup{(B)}.
\end{lem}
\begin{proof}
This improves on~\cite[Example 3]{GuiWal-transitive} which deals with the case
when $L$ is a point and $B$ is the zero section of $T\cpn$ or $T\hpn$. 
Consider the standard $S^1$-action on $TS^{2n+1}$ induced by the
diagonal $S^1$-action on $\C^{n+1}$. The action preserves
the $\R^{2n}$-subbundle $\nu$ of $TS^{2n+1}$ whose fibers are 
orthogonal to the \mbox{$S^1$-orbits}.
Let $E(\nu)$ be the total space of $\nu$, and 
$S\subset L\times E(\nu)$ be the preimage of $B$ under the orbit map 
$L\times E(\nu)\to L\times T\cpn$ that is the identity on the first factor. 
The circle bundle inclusion of $S^1\to S\to B$ into $S^1\to L\times E(\nu)\to L\times T\cpn$
is the identity on the fibers and a homotopy equivalence on the base, so 
the inclusion $S\to L\times E(\nu)$ is a homotopy equivalence (by the $5$-lemma
applied to the homotopy exact sequence of the bundles and the Whitehead theorem).
Arguing by contradiction assume that $\xi$ admits a connection with non-transitive holonomy.
Let $\nu_S$ be the normal bundle  to $S$ in $L\times E(\nu)$.
Since $\nu_S$ is a pullback of $\xi$, Lemma~\ref{lem: pullback} shows that
$\nu_S$ admits a connection with non-transitive holonomy. 
Thus (A) fails for $\nu_S$, and hence it also does for 
$L\times \nu$, see~\cite[Proposition 4.1]{BKS-mod1}. 
Remark~\ref{rmk: sphere 1-4} implies that 
$L\times \nu$ has a nowhere zero section, and hence so does $\nu$.
Thus $S^{2n+1}$ has two orthonormal tangent vector fields 
which is impossible for even positive $n$~\cite[Example 4L.5]{Hat-book}.
With obvious modifications the same proof works for $\hpn$.
\end{proof}

Finally, we give several examples where (B) fails.

\begin{ex} 
\label{ex: CPodd} 
The $S^1$ quotient of the Hopf fibration $S^{4m+3}\to \mathbb{H}\mathrm{P}^m$ is
a smooth fiber bundle $S^2\to \mathbb{C}\mathrm{P}^{2m+1}\to\mathbb{H}\mathrm{P}^m$.
Thus $T\mathbb{C}\mathrm{P}^{2m+1}$ splits as a Whitney sum if $m>0$, and hence fails (C).
I do not know if $T\mathbb{H}\mathrm{P}^{2m+1}$  satisfies (B) for $m>0$.
\end{ex}

\begin{ex}
The Riemannian connection on
the tangent bundle to a closed simply-connected irreducible compact
symmetric space $M$ has transitive holonomy if and only if $M$ has rank one, 
see~\cite[Sections 10.35, 10.79, 10.80]{Bes-book}; thus (B) fails for $TM$ if 
$M$ has rank $\ge 2$. 
In Example~\ref{ex: CPodd} we saw that (B) can fail (for some other metric connection)
even when the rank is one. 
\end{ex}

\begin{ex} (C) fails for the product of any two vector bundles with positive-dimensional fibers. 
If the factors have nonzero Euler classes, then $\xi$ satisfies (D), and hence 
$(D)\centernot\Longrightarrow (C)$. For example, 
the tangent bundle to $S^2\times S^2$
has no section (because the Euler characteristic is nonzero) but it splits as a Whitney sum. 
\end{ex}

\begin{ex} Let us show that 
$\mathrm{(C)}\centernot\Longrightarrow \mathrm{(B)}$.
Let $L$ denote any lens space 
of any (odd) dimension $\ge 3$ with $\pi_1(L)\cong \Z_m$, see~\cite[Example 2.43]{Hat-book}.
The universal coefficients theorem 
gives isomorphisms $H^2(L;\Z)\cong H^2(L;\Z_m)\cong\Z_m\cong
H^1(L;\Z_m)$. Denote  generators of $H^2(L;\Z)$, $H^1(L;\Z_m)$
by $\a$, $\b$, respectively. Let $\xi$ be an oriented $\R^2$-bundle over $L$ with Euler class $\a$.
Since $\a$ is torsion, $\xi$ is flat~\cite[Theorem 6.1]{KamTon}, 
and its holonomy group is an image of $\Z_m$, and hence (B) fails. 
Let us show that if $m=4k$ with $k\in\N$, then $\xi$ satisfies (C). 
By \cite[Example 3.41]{Hat-book} 
$2k\a (\mathrm{mod}\ m)=\b^2$ and further reducing mod $2$ gives $0=\b^2(\mathrm{mod}\ 2)$. 
Since $\b$ is a generator, the cup-square of any element of $H^1(L;\Z_2)$
is a multiple of $\b^2 (\mathrm{mod}\ 2)=0$.
On the other hand, $\a (\mathrm{mod}\ 2)$
generates a group isomorphic to $\Z_m\otimes \Z_2\cong\Z_2$,  
hence $\a (\mathrm{mod}\ 2)$ is not a square.
Now $\a (\mathrm{mod}\ 2)=w_2(\xi)$
If $\xi=\g\oplus\g^\prime$, 
the Whitney sum of two line bundles, then
since $\xi$ is orientable we get $0=w_1(\xi)=w_1(\g)+w_1(\g^\prime)$, so $\g=\g^\prime$
because an $\R$-bundle is determined by its $w_1$.
The Whitney sum formula gives $0\neq w_2(\xi)=w_1(\g)^2=0$, which is a contradiction. 
\end{ex}

\begin{lem}
\label{lem: flat soul exp not diffeo}
If $V$ is an open complete $n$-manifold of $K\ge 0$
with a flat soul and $H_{n-1}(V;\Z_2)=0$, then
every path-component of $\mathcal{R}_{K\ge 0}(V)$ 
contains a metric for which  
the normal exponential map to  soul is not a diffeomorphism
and the normal holonomy group of the soul is finite. 
\end{lem}
\begin{proof}
By~\cite[Corollary 6.3]{Wil-fund-gr} 
any metric can be joined by a path in $\mathcal{R}_{K\ge 0}(V)$ 
to a metric whose pullback to a finitely-sheeted Galois 
cover splits as a flat torus and the standard $\R^k$.
Here $k\ge 2$ because the soul has codimension $\ge 2$.
The action of the deck-transformation group $F$ on the $\R^k$-factor
is orthogonal, and one can $F$-equivariantly deform $\R^k$
to a rotationally symmetric metric $h$ that has  a cylindrical end
and positive curvature near the origin (a point fixed by $F$).
Then similarly to~\cite[page 77]{GroMey}
one can $F$-equivariantly perturb $(\R^k, h)$ to a complete metric
of $K\ge 0$ with no poles.
Multiplying $h$ by the above flat torus, and quotienting by $F$
gives a metric with claimed properties.
\end{proof}

Recall that a vector bundle is {\em spin\,} if and only if its $w_1$, $w_2$ vanish,
where $w_i$ is the $i$th Stiefel-Whitney class, and a manifold is {\em spin\,}
if so is its tangent bundle.
The following lemma lets us easily compute $w_i$ of the total space of $\xi$, or 
of the double of its unit disk bundle.

\begin{lem}
\label{lem: double SW}
Let $\xi$ be a smooth vector bundle over a closed
manifold $B$, and 
let $S$ be the double of the unit disk bundle of $\xi$ along the boundary.
Then $w_i(S)=0$ if and only if $w_i(\xi\oplus TM)=0$ if and only if $w_i$
of the total space of $\xi$ is zero. 
\end{lem}
\begin{proof}
The key point is that $S$ is the unit sphere bundle 
of $\xi\oplus\e^1$ where $\e^1$ is the trivial $\R$-bundle over $B$.
Let $\bm{q}\co D\to B$ be the unit disk bundle of $\xi\oplus\e^1$ and
set $q=\bm{q}\vert_S\co D\to B$, the sphere bundle projections.
Then $TD=\bm{q}^*(\xi\oplus TB)$, and 
since $\bm q$ is homotopic to the identity,
$w_i(\xi\oplus TB)=0$ if and only if $w_i(\Int\, D)=0$
where $\Int\, D$ is the total space of $\xi$.
Since $TS\oplus\e^1\cong TD\vert_S=q^*(\xi\oplus TB)$, 
the map $q^*$ takes $w_i(\xi\oplus TB)$ to $w_i(S)$.
Since $\xi\oplus\e^1$
has a nowhere zero section, the \mbox{$\Z_2$-Euler} class of $q$
is zero, hence the $\Z_2$-Gysin sequence of $q$ gives injectivity of 
$q^*\co H^i(M;\Z_2)\to H^i(S;\Z_2)$.
Thus $w_i(\xi\oplus TB)=0$ if and only if $w_i(S)=0$.
\end{proof}

\begin{ex}
\label{ex: RPn}
The total space of the tautological line bundle over 
$\mathbb{R}\mathrm{P}^n$ is spin if and only if $n\equiv 2\,(\mathrm{mod}\, 4)$,
as easily follows from the fact that $T\mathbb{R}\mathrm{P}^n$ is stably isomorphic to the 
Whitney sum of $n+1$ copies of the tautological bundle.
\end{ex}

\section{Approximating by a family of convex hypersurfaces}
\label{sec: family of hypersurfaces}

In this section we prove Theorem~\ref{thm: push to C}.
Since the normal sphere bundle to some (and hence any~\cite[Proposition 4.1]{BKS-mod1}) 
soul has no section, every metric $g\in\mathcal R_{K\ge 0}(V)$ has a unique
soul $S=S_g$ which depends continuously on $g$~\cite[Theorem 2.1]{BFK}.

Since the normal exponential map to $S$ is a diffeomorphism,
the distance function to the $S$ is smooth away from $S$. 
The function is also 
convex by Riccati comparison, see~\cite[Lemma 1.2]{EscFre}.
In particular, the closed $1$-neighborhood $D$ of $S$
is a totally convex, compact, smooth codimension zero submanifold of $V$ with infinite
normal injectivity radius. The distance
function $d(\cdot , D)$ to $D$ is smooth away from $D$, and also convex,
again, by Riccati comparison~\cite[Lemma 3.4(a)]{Esc-nneg}, or
alternatively, by concavity of $d(\cdot , \d D)$ on $D$
established in~\cite[Theorem 1.10]{CheGro-soul}.


We are going to approximate $(V,g)$ by a family of convex hypersurfaces in $V\times\R$. 
To visualise the process think of the surface in $\R^3$
that is a rotationally symmetric smoothing of the ``drinking glass''
\[
\{(x,y,0)\in \R^3\,:\, x^2+y^2\le r^2\}\cup \{(x,y,z)\in \R^3\,:\, x^2+y^2=r^2\ \text{and}\ z\ge 0 \}
\]
which tends to the $xy$-plane as $r\to\infty$.
Let $h\co (-\infty, 0)\to [0,1)$ be a surjective $C^\infty$ function such that 
\begin{itemize}
\item $h\vert_{(-\infty, -1]}=0$, 
\item the derivatives
$h^\prime$, $h^{\prime\prime}$ are positive on $(-1,0)$, 
\item
the inverse of $h\vert_{(-1,0)}$ extends to
a $C^\infty$ function $\bm h\co (0,\infty)\to (-1,0]$ that
vanishes on $[1, \infty)$.
\end{itemize}

Fix $r>1$. Let $N_r$ denote the closed $r$-neighborhood of $D$ in $V$.
For $\r\in [0,r)$
let $G_\r\subset V\times\R$ be the graph of $h(d(\cdot, D)-r)$ 
over $\mathrm{Int}\,N_{\r}$, and let  $\bar G_\r$ be its closure. 
Thus $\bar G_r\setminus G_r=\d N_{r}\times\{1\}$. 
The function $h(d(\cdot, D)-r)$ is convex because $d(\cdot, D)-r$ is convex and
$h$ is non-decreasing convex. Hence the induced metric on $G_r$ has $K\ge 0$,
see Lemma~\ref{lem: convex implies nonnnegative curv} below. 
The same holds for $\d N_{r}\times [1,\infty)$
because $\d N_{r}$ bounds $N_r$,
a convex subset of a nonnegatively curved manifold.
The union $H_r$ of $G_r$ and $\d N_{r}\times [1,\infty)$ is clearly
a topological submanifold of $V\times\R$ which is smooth away from 
$\d N_{r}\times \{1 \}$.
Lemma~\ref{lem: H_r is smooth} below
shows that $H_r$ is smooth, and hence its induced metric
also has $K\ge 0$.

To proceed we need some notations.
Let $U$ denote the $g$-unit normal bundle to $\d D$ so that
$V\setminus D$ is identified with $U\times (0,\infty)$ via
the normal exponential map $\exp^\bot$ of $D$. 
Let $f_r\co U\times (0,\infty)\to V\times\R$ be the map
given by  $f_r(u,t)=(\exp^\bot(\bm{h}_r(t)u), t)$ where
$\bm{h}_r:=\bm{h}+r$ takes values in $(r-1,r]$.
We now verify smoothness of $H_r$.

\begin{lem}
\label{lem: H_r is smooth}
$f_r$ is a smooth embedding whose image is $H_r\setminus\bar G_{r-1}$.
\end{lem}
\begin{proof}
Using that $\exp^\bot$ is a diffeomorphism and $U$ is compact one easily sees 
that $f_r$ is a smooth homeomorphism onto its image. Also $f_r$ is an
immersion because if $q$ is given by $q(\exp^\bot(v), t)=\left(\frac{v}{\|v\|}, t\right)$, 
where $\exp^\bot(v)\in V\setminus D$ and
$\|\cdot\|$ is the norm induced by the Riemannian metric, then 
$q\circ f_r$ is the identity. Thus $f_r$ is a smooth embedding.
For $t\ge 1$ we have $f_r(u,t)=(\exp^\bot(ru), t)$, so
$f_r$ maps $U\times [1,\infty)$ diffeomorphically onto 
$\d N_{r}\times [1,\infty)$.
If $0<t<1$, then $\bm{h}_r(t)\in (r-1, r)$,
and $x=\exp^\bot(\bm{h}_r(t)u)$ satisfies
$d(x,D)=\bm{h}_r(t)$, which can be rewritten as $t=h(d(x,D)-r)$.
Thus for each $\r\in (r-1,r)$ the map
$f_r$ takes $U\times\{h(\r-r)\}$ onto $\d G_\r$,
and therefore, 
$U\times (0,1)$ onto $G_r\setminus\bar G_{r-1}$. 

Next we pullback the metric on $H_r$ to $V$ in a way that is continuous in $g$ and $r$.

\begin{lem}
\label{lem: diff stretch}
There is a diffeomorphism $s_{r}\co V\to H_r$ which restricts 
to the identity on $G_{r-1}$ and depends continuously on $r$.
\end{lem}
\begin{proof}
The idea is to map $H_r$ diffeomorphically onto $G_{r-\frac{1}{3}}$
by contracting along the
geodesics orthogonal to $\d D$ so that the contraction is the identity on
$G_{r-\frac{1}{2}}$, and then project orthogonally onto $N_{r-\frac{1}{2}}$,
and finally stretch radially to $V$.
To do this continuously set $a=h(-1/2)$ and $b=h(-1/3)$.
Let $D_{a,b}$ be the diffeomorphism of
$U\times (0,\infty)$ onto $U\times (0,b)$ 
given by $D_{a,b}(u,t)=(u, d_{a,b}(t))$
where $d_{a,b}$ is as in Lemma~\ref{lem: stretch reals} below. 
Then $f_r\circ D_{a,b}\circ f_r^{-1}$ is a diffeomorphism of
$H_r\setminus\bar G_{r-1}$ onto $G_{r-\frac{1}{3}}\setminus\bar G_{r-1}$
that is the identity on $G_{r-\frac{1}{2}}\setminus\bar G_{r-1}$.
Extend $f_r\circ D_{a,b}\circ f_r^{-1}$ 
to a diffeomorphism $H_r\to G_{r-\frac{1}{3}}$ 
that is the identity on $\bar G_{r-1}$.
Postcompose the result with the coordinate projection $V\times \R\to V$
and then with the self-diffeomorphism of $V$ given by 
$\exp^\bot(v)\to\exp^\bot
\left(d_{\a,\b}(\|v\|)\frac{v}{\|v\|}\right)$.
This gives a diffeomorphism $H_r\to V$ 
whose inverse has desired properties.
\end{proof}

Let $g[r]$ denote the the pullback  of the induced metric on $H_r$ via $s_r$.
For $t\in [0,1]$ define a self-map $\r_t$ of $\mathcal R_{K\ge 0}(V)$
by setting $\r_t(g):=g\!\left[\frac{1}{t}\right]$ for $t>0$, and $\r_0(g):=g$.  

\begin{lem}
The map $(t,g)\to\r_t(g)$ is continuous.
\end{lem}
\begin{proof}
Given a sequence $(t_i,g_i)$ converging to $(t,g)$ and
a compact subset $K$ of $V$ we are to show that $\r_{t_i}(g_i)$ 
converges to $\r_t(g)$ on $K$. If $t>0$ this follows because by construction 
$g[r]$ depends continuously on $g$, $r$. 
Suppose $t=0$, so that $\r_{t}(g)=g$. Since $S_{g_i}$ converge to $S_g$,
the $\left(\frac{1}{\,t_i}-1\right)$-neighborhood of $S_{g_i}$ contain $K$ for all large $i$, and hence
over $K$ we have 
$g_i\!\!\left[\frac{1}{\,t_i}\right]=g_i\to g$ as $i\to\infty$. 
\end{proof}

The distance in $H_r$ between the soul and the cylindrical end $\d N_r\times [1,\infty)$
is realized by a geodesic that
lies in a $2$-dimensional half-flat $\s\times \R$, where $\s$ a ray in $V\times\{0\}$ 
orthogonal to the soul. The portion of the geodesic in $N_{r-1}\times\{0\}$ has length $r-1$ and the remaining portion can be identified with the graph of $h$ over $[-1, 0)$ whose length is at most $2$.
Setting $r=\frac{1}{\e}$ gives the last assertion of Theorem~\ref{thm: push to C}. 
\end{proof}

\begin{rmk}
If $(V, g)$ has a cylindrical end and $r$ is sufficiently large, then
$g[r]=g$. More precisely, if $\r\in (0,r-1)$ and
$(V\setminus N_\r, g)$ is isometric to the product of a closed manifold
and $(0,\infty)$, then $g[r]=g$
because $H_r$ and $V$ are both obtained by gluing $N_{\r}$ via the identity map of 
$\d N_{\r}$ to the Riemannian product of $\d N_{\r}$ and
the image of a smooth proper embedding of $[0,\infty)$ 
into $\R\times [0,\infty)$, and the diffeomorphism $s_r$ of Lemma~\ref{lem: diff stretch}
identifies $g$ and the induced metric on $H_r$.
\end{rmk}

\begin{rmk}
\label{rmk: diff equiv}
It follows by construction that 
the map $g\to g[r]$ is $\Diff V$-equivariant, i.e., if $\phi\in\Diff V$,
then $(\phi^*g)[r]=\phi^*(g[r])$.
\end{rmk}

Finally, here are two elementary lemmas that were used above.

\begin{lem}
\label{lem: convex implies nonnnegative curv}
If $V$ is a Riemannian manifold of $K\ge 0$ and $f\co V\to\R$ is convex and smooth,
then the graph of $f$ has nonnegative sectional curvature
in the metric induced by Riemannian product $V\times\R$.
\end{lem}
\begin{proof}
By Gauss equation any smooth convex hypersurface in a manifold of $K\ge 0$
has nonnegative sectional curvature. Since the graph of $f$ is a level set of 
the function $F(x,t)=f(x)-t$, it suffices to check that $F$ is convex, or equivalently,
that the Hessian of $F$ is positive semidefinite. Fix normal coordinates centered
at an arbitrary point of $p\in V\times\R$, where the last coordinate 
is the projection on the $\R$-factor. Then the $ij$ entry of $\mathrm{Hess}\,F$ 
at $p$  equals $\d_i\d_j f$ if $i,j\le\dim V$
and zero otherwise, so convexity of $f$ implies convexity of $F$.
\end{proof}

\begin{lem}
\label{lem: stretch reals}
Given $0<a<b$ there is a $C^\infty$ diffeomorphism 
$d_{a,b}\co (0,\infty)\to (0,b)$ 
which restricts to the identity on $(0,a]$, and varies continuously with $a$ and $b$
in the space of smooth self-maps of $(0,\infty)$. 
\end{lem}
\begin{proof}
Fix a bump function 
supported on $[-1, 1]$, apply an affine change of variable to produce a 
bump function with support $[a,\frac{a+b}{2}]$, 
integrate it from $0$ to $x$ and scale to get a nonnegative function that
is zero on $(-\infty, a]$ and $1$ on $[\frac{a+b}{2}, \infty)$.
Multiply it on $(0,b)$ 
by $\frac{2}{(b-x)^3}$ and denote the result by $q_{a,b}$.
The solution of $f^{\prime\prime}=q_{a,b}$, $f(0)=0$, $f^\prime(0)=1$
satisfies $f(x)=x$ for $x\le a$, and if $\frac{a+b}{2}\le x<b$,
then $f(x)=\frac{1}{b-x} + cx+d$ for some constants $c$, $d$
continuously depending on $a$, $b$.
Now $q_{a,b}\ge 0$ and $f^\prime(0)=1$ gives $f^\prime\ge 1$,
and $d_{a,b}:=f^{-1}$ has the desired properties. 
\end{proof}

\section{Proof of main results}
\label{sec: main}

In this section we prove Theorem~\ref{thm: intro-main} and Corollary~\ref{cor: intro applications}.
Given $r>0$ and $V$ as in Theorem~\ref{thm: push to C}
consider the set of pairs $(g,x)$ such that $g\in \mathcal R_{K\ge 0}(V)$
and $x\in D_{r,g}$, where $D_{r,g}$ is the closed $r$-neighborhood of the (unique) soul of $g$.
Since $D_{r,g}$ varies continuously with $g$, 
the map $\mathfrak d_r(g, x)=g$ is a smooth fiber bundle whose fiber over $g$ is $D_{r,g}$. 
The normal exponential map to the soul of $g$ identifies $D_{r,g}$ with $D_{1,g}$,
which gives rise to an isomorphism $\mathfrak d_r\cong \mathfrak d_1$.
While this fiber bundle is not really needed for the proof, it seems to
illuminate our strategy.

Let $\mathcal{D}_0(M)$ be the subspace of $\Diff\, M$ consisting of compactly-supported
diffeomorphisms that are isotopic to the identity.
Let $\mathcal R_{\mathrm{scal}\gneqq 0}(M)$ be  
the space of complete metrics on  $M$
of nonnegative but not identically zero scalar curvature.

\begin{thm}
\label{thm: main technical}
Let $V$ be as in \textup{Theorem~\ref{thm: intro-main}(ii)} except that
$V$ can be non-spin.
Let $Y$ be a finite contractible CW complex $Y$ and $z\in Y$.
Let $f\co Y\to \mathcal R_{K\ge 0}(V)$ be a continuous map. 
Let $X\subset Y$ be a subcomplex and  
let $\de\co X\to\mathcal D_{0}(V)$ be a continuous map 
such that $f(x)$ is the pullback of $f(z)$ via $\de(x)$ for all 
$x\in X$. Let $D$ be a closed tubular neighborhood of the soul of $f(z)$
such that the support of every diffeomorphism in $\de(X)$ lies in the interior of $D$.
Let $\widehat D$ be the double of $D$, and let $\widehat\de\co X\to \Diff(\widehat D)$
that extends $\de$ by the identity outside $D$. 
Then there is a continuous map $\widehat f\co Y\to\mathcal R_{\mathrm{scal}\gneqq 0}(\widehat D)$
such that $\widehat f(x)$ is the pullback of $\widehat f(z)$ 
via $\widehat\de(x)$ for all $x\in X$.
\end{thm}

\begin{proof}
{\bf Step 1: Push to metrics with cylindrical ends of definite size.}
The homotopy 
$\r_\e$ deforms $f$ into $\mathcal C_{K\ge 0}(V)$, so after 
changing $f$ to a nearby homotopic map we can assume that
each metric in the image of $f$
is a product outside the $\left(1+\frac{1}{\e}\right)$-neighborhood of its soul. 
Also assume that $\e$ is so small that the interior of the neighborhood
contains the support of every diffeomorphism in $\de(X)$. 
Remark~\ref{rmk: diff equiv} ensures that after the homotopy 
the map $f\vert_X$ is still induced by $\de$.

{\bf Step 2: Preparing the neck topologically.}
For $y\in Y$ set $g_y:=f(y)$, let $S_y$ be
the soul of $g_y$, let $D_{r, y}$ denote the $r$-neighborhood of $S_y$
in the metric $g_y$, and $\Si_{r,y}=\d D_{r,y}$. Note that $D_{r, y}$ 
depends continuously of $y$, $r$.

Fix $\r>1+\frac{1}{\e}$ and set $D_y:=D_{\r, y}$, $\Si_y=\d D_y$.
Since $Y$ is contractible, the pullback of 
$\mathfrak d_\r$ via $f$ is trivial, which defines a continuous map
$Y\to \mathrm{Emb}(D_{z}, V)$ that sends $y$ to the diffeomorphism $D_z\to D_y$.

The map $\mathcal{D}_0(V)\to \mathrm{Emb}(D_{z}, V)$ 
that precomposes a diffeomorphism  with the inclusion $D_{z}\hookrightarrow V$ is a fiber bundle
over the path-component of the inclusion~\cite{Pal}. Its fiber consists of diffeomorphisms
in $\mathcal{D}_0(V)$ that are the identity on $D_z$, and the space of such
diffeomorphisms is contractible
by the Alexander's trick towards infinity.
Thus the above map $Y\to \mathrm{Emb}(D_{z}, V)$ 
lifts to a map $Y\to\mathcal{D}_0(V)$, to be written as $y\to\phi_y$,
such that $\phi_{z}$ is the identity. Thus
$\phi_y(D_{z})=D_y$. 

For the metric $\phi_y^*g_y$ we 
consider the normal bundle $\nu_y$ to $\Si_{z}$ in $V$, 
the normal exponential map $e_y\co \nu_y\to V$, 
and the unit normal vector field $U_v$ to $\Si_z$ pointing outside $D_{z}$.
Let $L_y\co \nu_{z}\to \nu_y$
be the (unique) linear isometry taking $U_{z}$ to $U_y$.
Since $\phi_y^*g_y$ is a product metric outside $D_z$, the map
$e_{y}\circ L_y\circ e_{z}^{-1}$ is a self-diffeomorphism of
$V\setminus \Int D_{z}$ which fixes every point of $\Si_z$.
Restricting the composite $e_{y}\circ L_y\circ e_{z}^{-1}$ to
the $1$-neighborhood $N$ of $\Si_{z}$ in $(V\setminus \Int D_{z}, g_z)$
gives a family of embeddings of $N$  into $V\setminus \Int D_{z}$
that restrict to the identity on $\Si_{z}$.
As in the previous paragraph the space of 
diffeomorphisms of $V\setminus \Int D_z$ with compact support that are identity on $N$
is contractible, so by~\cite{Pal} the family of embeddings can be obtained by 
restricting a continuous family of diffeomorphisms $\psi_y\in\mathcal D_0(V)$
such that $\psi_z$ is the identity.

An alternative way to extend isotopies in the above two paragraphs is to 
use contractibility of $Y$. 
In summary, the diffeomorphism $\iota_y:=\phi_y\circ\psi_y$ 
has compact support, and maps $\Si_{z}$, $U_{z}$ to $\Si_y$, $U_y$, respectively.
Also $\iota_z=\mathrm{id}$. 
By compactness of $Y$ there is a compact subset of $V$
outside which every diffeomorphisms $\iota_y$
equals the identity. 
By construction the $\iota_y^*g_y$-distance $r$ from a point of $N$ to $\Si_{z}$
in independent of $y$ (the induced metrics on the fibers of $r$
may depend on $y$).

{\bf Step 3: Deforming metrics on the neck.}
Let $b_y$ be the metric induced by $\iota_y^*g_y$ on $\Si_y$.
Since $Y$ is contractible, there is a deformation retraction $q_s\co Y\to Y$, $s\in [0,1]$,
where $q_1$ is the constant map to $z$ and $q_s$ is the identity for $s\in [0,\frac{1}{3}]$. 
Then $b_{q_s(y)}$ is a homotopy through maps
$Y\to \mathcal R_{K\ge 0}(\Sigma)$
between $y\to b_y$ and the constant map to $b_{z}$.
Applying Ricci flow for time $t$
with the initial metric $b_{q_s(y)}$ yields a map $(t,s,y)\to b_{q_s(y)}^{\,t}$
which is continuous by the smooth dependence of the Ricci flow under the initial data.

The normal sphere bundle to the soul is not flat, see Lemma~\ref{lem: flat soul} below.
Hence a result in~\cite{BohWil} stated in Lemma~\ref{lem: Bohm-Wilking nonnegative} below, 
shows that $b_{q_s(y)}^{\,t}$ has positive scalar curvature for all small positive $t$.
Compactness of $Y$ gives $\tau>0$ such that
the flow exists and has positive scalar curvature 
for all $t\in [0,\tau]$, $y\in Y$, $s\in [0,1]$.

Fix a continuous map $\mu\co [0,1]\to [0,\tau]$
with $\mu^{-1}(0)=[0,\frac{1}{3}]$ and $\mu^{-1}(\tau)=\left[\frac{1}{2}, 1\right]$. 
Then $s\to b_{q_s(y)}^{\,\mu(s)}$ is a homotopy from $y\to b_y$ to $y\to (b_z)^\tau$
such that $b_{q_s(y)}^{\,\mu(s)}$ equals $b_y$, $b_z^\tau$ for $s$ in $\left[0,\frac{1}{3}\right]$, 
$\left[\frac{1}{2}, 1\right]$,
respectively, and $b_{q_s(y)}^{\,\mu(s)}$ has positive scalar curvature 
on $\left(\frac{1}{3}, 1\right]$.

By~\cite[Lemma 3]{PetYun} for all sufficiently large $a$ the metric 
$ds^2+b_{q_{_{s/a}}\!(y)}^{\,\mu(s/a)}$  has
positive scalar curvature on $\left(\frac{a}{3},a\right]\times \Si_{z}$,
and nonnegative sectional curvature on $\left[0,\frac{a}{3}\right]\times \Si_{z}$
where $b_{q_{_{s/a}}\!(y)}^{\,\mu(s/a)}=b_y$
for $s\in \left[0, \frac{a}{3}\right]$ 
and $b_{q_{_{s/a}}\!(y)}^{\,\mu(s/a)}=(b_{z})^\tau$ for
$s\in \left[\frac{a}{3}, a\right]$.

Let us pick $a$ so large that $a>1$ and the union of the supports of $\iota_y$ over $y\in Y$
lies in the interior of $D_{\r+a, z}$.

The metrics $\iota_y^*g_y$ on $D_{\r+\frac{1}{3}, z}$ 
and $dr^2+(b_{q_{_{r/a}}\!(y)})^{\mu(r/a)}$ on $D_{\r+a, z}\setminus D_z$
restrict to the metric $dr^2+b_y$ on $D_{\r+\frac{1}{3}, z}\setminus D_z$ 
where $r$ is the $g_{z}$-distance to $\Si_{z}$, so together they define a metric
on  $D_{\r+a, z}$ which we denote  $\bar g_y$.

{\bf Step 3: Doubling.}
The pullback of $\bar g_y$ via $\iota_y^{-1}$
is a metric on $D_{\r+a, z}$ such that $\bar g_y=g_y$ on $D_y$
and $\bar g_y=dr^2+(b_z)^\tau$ near $\Si_{\r+a, z}$.
If $\widehat{D}_{\r+a, z}$  is the double of $D_{\r+a, z}$ 
such that the doubling involution $j$ preserves every geodesic orthogonal to
$\Si_{\r+a, z}$, then $\bar g_y$ and $j^*\bar g_z$ form a smooth metric on $\widehat{D}_{\r+a, z}$.
Denote the metric by $\widehat f(y)$. 
Then the map $y\to \widehat f(y)$ is continuous, and for every $x\in X$
the restriction $\widehat f(x)\vert_{D_z}$ is the pullback of 
$g_z=\widehat f(z)\vert_{D_z}$ via $\de(x)$, while 
$\widehat f(x)=\widehat f(z)$ on $\widehat{D}_{\r+a, z}\setminus D_z$.
Thus $\widehat f$ has desired properties.
\end{proof}

\begin{rmk}
\label{rmk: codim 1}
The conclusion of Theorem~\ref{thm: main technical} is true 
if the assumptions on $V$ are replaced with ``$V$ is the product of $\R$
and a closed non-flat manifold of $K\ge 0$''.
In this case the souls are precisely the fibers of the projection onto the $\R$-factor.
Given a basepoint $*\in V$ every metric  
$g\in\mathcal R_{K\ge 0}(V)$ has a unique soul that contains $*$, and 
the soul depends continuously on $g$ (because the local
$\R$-factor is uniquely determined by the metric, and the limit
of local $\R$-factors is a local $\R$-factor).
Since both ends of $V$ are already cylindrical, Step 1 is unnecessary.
The rest of the proof goes through without change. 
\end{rmk}

\begin{lem}
\label{lem: Bohm-Wilking nonnegative}
Let $(M, g)$ be a closed manifold of $K\ge 0$ that is not flat.
Then  Ricci flow $g_t$ with $g_0=g$ has positive scalar curvature for all small positive $t$.
\end{lem}
\begin{proof}
According to~\cite[Proposition 2.1 and Theorem A]{BohWil} for all small positive $t$
the metric $g_t$ has nonnegative Ricci curvature, which is positive 
if and only if $\pi_1(M)$ is finite. The universal cover $\widetilde M$
of $M$ splits as the product of a Euclidean space and a closed 
simply-connected manifold of $K\ge 0$, and the latter has dimension $\ge 2$
since $M$ is not flat~\cite{CheGro-soul}.
The pullback $\widetilde{\! g}_t$ of $g_t$ to $\widetilde M$ is a solution of Ricci flow.
The uniqueness of Ricci flow solutions with complete initial metrics of bounded curvature~\cite{CheZhu-uniq}
implies that the flow coincides with the product of Ricci flows on the 
factors. The flow leaves the Euclidean factor unchanged, and instantly 
turns the initial metric on the simply-connected factor to a metric of positive Ricci curvature,
thanks to the above-mentioned result in~\cite{BohWil}.
In particular, $\widetilde{\! g}_t$ has positive scalar curvature, and hence so does $g_t$.
\end{proof}

\begin{proof}[Proof of Theorem~\ref{thm: intro-main}]
Fix a smoothly embedded $n$-disk in $V$. Extending by the identity
we think of $\Diff (D^n,\d)$ as a subgroup of $\mathcal D_0(V)$.
Fix an order two element of $\pi_k(\Diff (D^n,\d))$ detected by the $\a$-invariant 
as in~\cite{CSS}, and represent it by a continuous map $\de\co S^k\to \mathcal D_0(V)$.
Consider the map $S^k\to \mathcal R_{K\ge 0}(V)$
that sends $x$ to the pullback of $g$ via $\de(x)$. Its homotopy class 
has order at most two. Arguing by contradiction assume the map extends to a map from $D^{k+1}$.
By assumption the soul is not flat,
hence the normal sphere bundle to the soul
is not a flat manifold, see Lemma~\ref{lem: flat soul} below.
Also the double of any tubular neighborhood of the soul is spin 
because so is $V$, see Lemma~\ref{lem: double SW}.
Apply Theorem~\ref{thm: main technical} or Remark~\ref{rmk: codim 1}
for $(Y, X)=(D^{n+1}, S^k)$
to conclude that the map sending $x\in S^k$ to the pullback of $\widehat g$ via  $\widehat\de(x)$ 
is null-homotopic. By~\cite{CSS} the $\a$-invariant of the map is nontrivial,
which is a contradiction.
\end{proof}

\begin{rmk}
The proof of Theorem~\ref{thm: intro-main} fails when the soul is flat: 
if the flat soul has codimension one, then the double of the normal disk bundle
to the soul is flat, hence the Dirac operator is not invertible, and if
the flat soul has codimension $\ge 2$, 
its normal exponential map need not be a diffeomorphism
by Lemma~\ref{lem: flat soul exp not diffeo}.
\end{rmk}

\begin{lem} 
\label{lem: flat soul}
Let $B$ be a soul of an open connected complete $n$-manifold of $K\ge 0$.
Then its normal sphere bundle $E$ is flat if and only if $B$ is flat and $\dim(B)\ge n-2$.
\end{lem}
\begin{proof}
By~\cite{CheGro-soul} a closed connected manifold
of $K\ge 0$ has a finite cover diffeomorphic to the product of 
a simply-connected closed manifold and a torus. Any connected flat
manifold is finitely covered by a torus. 
Write the normal sphere bundle to the soul as $S^k\to E\to B$ and refer to its 
homotopy exact sequence as HES.
Here $\dim(E)=n-1=k+\dim(B)$, and $E$, $B$ carry metrics of $K\ge 0$
(for $E$ this is due to~\cite{GuiWal-proj}).
If $B$ is flat and $k\le 1$, then HES implies that $E$ is aspherical, and hence flat.
Conversely, if $E$ is flat, then $\pi_1(E)$ is virtually-$\Z^{n-1}$.
The  portion 
$\pi_1(S^k)\to \pi_1(E)\to\pi_1(B)\to\pi_0(S^k)$
of HES
implies that $\pi_1(B)$ is virtually-$\Z^m$ with $m\ge n-2$.  
Now $m\le \dim(B)$ gives $n-2\le n-k-1$, i.e., $k\le 1$.
Then HES shows that $B$ is aspherical, and hence flat. 
\end{proof}

\begin{proof}[Proof of Corollary~\ref{cor: intro applications}]
Let us first discuss why the manifolds in (1)-(10) admit complete metrics of $K\ge 0$.
Having $K\ge 0$ is preserved under products, which covers (1) and (10).
Another basic method is to realize the manifold as a base of Riemannian submersion
from a complete manifold of $K\ge 0$, which is possible for tangent bundles
of homogeneous spaces~\cite[Example 2.7.1]{GroWal-book} and for
the tautological line bundles over projective spaces,
e.g., as $\cpn=(S^{2n+1}\times\C)/S^1$,
where in fact, by varying $S^1$-actions on $\C$ one gets all complex line bundles
over $\cpn$, see the proof of~\cite[Theorem 6.1]{BelWei-gafa}. 
This covers (2), (5)--(7), and the first example in (8). 
The other examples are due to K.~Grove and W.~Ziller:
(3)--(4) can be found in Theorem B, Corollary 3.13, Proposition 3.14 in~\cite{GroZil-milnor},
the remaining cases of (8) are treated in
Theorem 4.5 and the remark that follows it in~\cite{GroZil-lift},
and (9) appears in~\cite[Theorem 1]{GroZil-lift}.

Verifying that $V$ is spin is immediate from
Lemma~\ref{lem: double SW} and the Whitney sum formula for the 
Stiefel-Whitney class; for more details on (5) see 
Example~\ref{ex: RPn}.

Finally, let us check the assumptions (i) or (ii) of Theorem~\ref{thm: intro-main}.
The bundles in (1) satisfy (i) by Remark~\ref{rmk: codim 1}. The condition
(B) of Section~\ref{sec: transitive} implies (ii).
By Example~\ref{ex: product times L} we can assume that $L$ is a point.
The bundles in (3) have no section (because 
every $\R$ or $\R^2$ bundle over $S^n$ with $n\ge 4$ is trivial),
and hence by Lemma~\ref{lem: R3 bundles no section} they satisfy (B).   
By Remark~\ref{rmk: euler hurewicz}
the bundles in (2), (4)--(5), (7)--(8) 
satisfies (A), and hence (B). By Lemma~\ref{lem: TCPn THPn}
the bundles in (6) satisfy (B), and 
(9) is covered in Example~\ref{ex: SO(3) bundles over CP2}.
Invoking Theorem~\ref{thm: intro-main} completes the proof.
\end{proof}

\small
\bibliographystyle{amsalpha}
\bibliography{cyl}

\end{document}